\documentclass[11pt,reqno]{amsart}
\usepackage{color}

\usepackage{amsmath}
\usepackage{amsfonts,amscd}
\usepackage{amssymb}
\usepackage{url}
\usepackage{hyperref}

\usepackage[english]{babel}
\usepackage{booktabs}
\usepackage{tikz}

\theoremstyle{plain}
\newtheorem{theorem}                 {Theorem}      [section]

\newtheorem{corollary}    [theorem]  {Corollary}
\newtheorem{lemma}        [theorem]  {Lemma}
\newtheorem{proposition}  [theorem]  {Proposition}

\theoremstyle{definition}

\newtheorem{definition}   [theorem]  {Definition}

\newtheorem{example}      [theorem]  {Example}

\numberwithin{equation}{section}

\def \cn{{\mathbb C}}
\def \C{{\mathbb C}}
\def \hn{{\mathbb H}}

\def \rn{{\mathbb R}}
\def \R{{\mathbb R}}

\def \B{\mathcal B}
\def \E{\mathcal E}

\def\nab#1#2{\hbox{$\nabla$\kern -.3em\lower 1.0 ex
		\hbox{$#1$}\kern -.1 em {$#2$}}}
\def\hatnab#1#2{\hbox{$\nabla$\kern -.3em\lower 1.0 ex
		\hbox{$#1$}\kern -.1 em {$#2$}}}

\def \Re{\mathfrak R\mathfrak e}
\def \Im{\mathfrak I\mathfrak m}
\def \inv {^{-1}}

\def \Sym{\mathrm{Sym}}

\def \g{\mathfrak{g}}

\def \k{\mathfrak{k}}

\def \p{\mathfrak{p}}

\def\span{\mathrm{span}}

\def \GLR#1{\mathbf{GL}_{#1}(\rn)}

\def \glr#1{\mathfrak{gl}_{#1}(\rn)}
\def \GLC#1{\mathbf{GL}_{#1}(\cn)}
\def \glc#1{\mathfrak{gl}_{#1}(\cn)}

\def \SLR#1{\mathbf{SL}_{#1}(\rn)}
\def \SL2{\widetilde{\text{\bf SL}}_{2}(\rn)}
\def \slr#1{\mathfrak{sl}_{#1}(\rn)}
\def \SLC#1{\mathbf{SL}_{#1}(\cn)}
\def \slc#1{\mathfrak{sl}_{#1}(\cn)}

\def \SO#1{\mathbf{SO}(#1)}
\def \so#1{\mathfrak{so}(#1)}
\def \SOs#1{\mathbf{SO}^*(#1)}
\def \sos#1{\mathfrak{so}^*(#1)}

\def \SOC#1{\text{\bf SO}(#1,\cn)}

\def \soC#1{\mathfrak{so}(#1,\cn)}

\def \sus#1{\mathfrak{su}^*(#1)}

\def \U#1{\text{\bf U}(#1)}
\def \u#1{\mathfrak{u}(#1)}
\def \Us#1{\text{\bf U}^*(#1)}
\def \us#1{\mathfrak{u}^*(#1)}

\def \SU#1{\text{\bf SU}(#1)}

\def \SUs#1{\text{\bf SU}^*(#1)}
\def \sus#1{\mathfrak{su}^*(#1)}

\def \Sp#1{\text{\bf Sp}(#1)}
\def \sp#1{\mathfrak{sp}(#1)}

\def \SpR#1{\text{\bf Sp}(#1,\rn)}
\def \spR#1{\mathfrak{sp}(#1,\rn)}

\def \spC#1{\mathfrak{sp}(#1,\cn)}

\DeclareMathOperator{\Div}{div} 

\DeclareMathOperator{\trace}{trace}

\numberwithin{equation}{section}
\allowdisplaybreaks

\begin{document}

\subjclass[2020]{53C35, 53C43, 58E20}
	
\keywords{minimal submanifolds, eigenfunctions, symmetric spaces}

\author{Sigmundur Gudmundsson}
\address{Mathematics, Faculty of Science\\
Lund University\\
Box 118, Lund 221 00\\
Sweden}
\email{Sigmundur.Gudmundsson@math.lu.se}

\author{Lucas Larsen}
\address{Mathematics, Faculty of Science\\
	Lund University\\
	Box 118, Lund 221 00\\
	Sweden}
\email{Lukas.Larsen@math.ku.dk}

\title
[Complete Minimal Submanifolds of Symmetric Spaces]
{Complete Minimal Submanifolds of the Non-Compact Riemannian Symmetric Spaces\\ 
$\SLR n/\SO n$, $\SpR{n}/\U n$, $\SOs{2n}/\U n$, $\SUs{2n}/\Sp n$\\ 
via Complex-Valued Eigenfunctions}

\begin{abstract}
In this work we construct new multidimensional families of complete minimal submanifolds, of the classical non-compact Riemannian symmetric spaces $\SLR n/\SO n$, $\SpR{n}/\U n$, $\SOs{2n}/\U n$ and $\SUs{2n}/\Sp n$, of codimension two.
\end{abstract}
	
%\dedicatory{version 1.020 - \today\ - current editor: SG}

\maketitle

%%%%%%%%%%%%%%%%%%%%%%%%%%%%%%%%%%%%%%%%%%%
\section{Introduction}
\label{section-introduction}
%%%%%%%%%%%%%%%%%%%%%%%%%%%%%%%%%%%%%%%%%%%

The study of minimal submanifolds of a given ambient space plays a central role in differential geometry.  This has a long, interesting history and has attracted the interests of profound mathematicians for many generations.  The famous Weierstrass-Enneper representation formula, for minimal surfaces in three-dimensional Euclidean space, brings {\it complex analysis} into play as a useful tool for the study of these beautiful objects.

This was later generalised to the study of minimal surfaces in much more general ambient manifolds via {\it harmonic conformal immersions}.  The next  result follows from the seminal paper \cite{Eel-Sam} of Eells and Sampson from 1964.  For this see also Proposition 3.5.1 of \cite{Bai-Woo-book}.

\begin{theorem}
Let $\phi:(M^m,g)\to (N,h)$ be a smooth conformal map between Riemannian manifolds.  If $m=2$ then $\phi$ is harmonic if and only if the image is minimal in $(N,h)$.
\end{theorem}

This result has turned out to be very useful in the construction of minimal surfaces in Riemannian symmetric spaces of various types.  For this we refer to \cite{Cal},
\cite{Eel-Woo}, \cite{Uhl}, \cite{Bur-Raw} and \cite{Bur-Gue}, just to name a few.
\smallskip

In their work \cite{Bai-Eel} from 1981, Baird and Eells have shown that complex-valued harmonic morphisms from Riemannian manifolds are useful tools for the study of minimal submanifolds of codimension two. 

\begin{theorem}\cite{Bai-Eel}
\label{theorem:Bai-Eel-special}
Let $\phi:(M,g)\to\cn$ be a complex-valued harmonic morphism from a Riemannian manifold.  Then every regular fibre of $\phi$ is a minimal submanifold of $(M,g)$ of codimension {\it two}.
\end{theorem}

This can be seen as dual to the above-mentioned generalisation of the Weierstrass-Enneper representation. Harmonic morphisms are the much studied {\it horizontally conformal harmonic maps}.  For an introduction to the general theory we recommend the book \cite{Bai-Woo-book}, by Baird and Wood, and the regularly updated online bibliography \cite{Gud-bib}.

\section{The Main Results}

The recent work \cite{Gud-Mun-1} introduces a method for constructing minimal submanifolds of Riemannian manifolds  via {\it submersions}, see Theorem \ref{theorem-Gud-Mun-1}.  Then this scheme is employed to  provide {\it compact} examples in several important cases. The main ingredients for this new procedure are the so called {\it complex-valued eigenfunctions} on the Riemannian ambient space.  These are functions which are eigen both with respect to the classical {\it Laplace-Beltrami} and the so called {\it conformality} operator, see Section \ref{section-eigenfunctions}. In the recent study \cite{Geg-Gud-1} the authors continue the investigation and apply the above-mentioned method to the classical compact Riemannian symmetric spaces 
$$
\SU n/\SO n,\ \Sp{n}/\U n,\ \SO{2n}/\U n,\ \SU{2n}/\Sp n.
$$
%\smallskip 

In the current work we are concerned with a similar study in the classical non-compact dual Riemannian symmetric spaces 
$$
\SLR n/\SO n,\ \ \SpR{n}/\U n,\ \ \SOs{2n}/\U n,\ \  \SUs{2n}/\Sp n
.$$
For the first case we construct a real $(2n-2)$-dimensional family of {\it complete} minimal submanifolds of $\SLR n/\SO n$ of codimension two.

\begin{theorem}\label{theorem-SLRSO}
Let $n\geq 3$ and $a=a_1+ia_2\in\C^{n}$ be such that $a_1,a_2\in\R^{n}$ are linearly independent. Further let  $\phi:\SLR n/\SO n \to \C$ be the complex-valued eigenfunction defined by
$$
\phi(x\cdot \SO n) = \trace(aa^txx^t).
$$
Then the inverse image $\phi^{-1}(\{0\})$ is a complete minimal submanifold of $\SLR n/\SO n $ of codimension two.
\end{theorem}

For the Riemannian symmetric space $\SpR{n}/\U{n}$ we construct a real $(4n-2)$-dimensional family of {\it complete} minimal submanifolds of codimension two. 

\begin{theorem}\label{theorem-SpRU}
Let $n\geq 2$ and $a=a_1+ia_2\in\C^{2n}$ be such that $a_1,a_2\in\R^{2n}$ are linearly independent. Further let $\phi:\SpR n/\U n \to \C$ be the complex-valued map defined by
$$
\phi(x\cdot\U n) = \trace(aa^txx^t).
$$
Then the inverse image $\phi^{-1}(\{0\})$ is a complete minimal submanifold of $\SpR n/\U n)$ of codimension two.
\end{theorem}

In the third case we construct a real $(8n-10)$-dimensional family of {\it complete} minimal submanifolds of $\SOs{2n}/\U n$ of codimension two.

\begin{theorem}\label{theorem-SOsU}
Let $n\geq 2$ and $a,b\in\C^{2n}$ be linearly independent vectors such that 
$$
(a,a) = (a,b) = (J_na, b) = 0 \neq (b,b)
$$ 
Further let the complex-valued function $\phi_{a,b}:\SOs{2n}/\U n \to \C$ be defined by
$$
\phi_{a,b}(z\cdot\U n) = \trace(ab^tzJ_nz^t).
$$
Then the inverse image $\phi_{a,b}^{-1}(\{0\})$ is a complete minimal submanifold of $\SOs{2n}/\U n$ of codimension two.
\end{theorem}

For the Riemannian symmetric space $\SUs{2n}/\Sp{n}$ we construct a real $(4n-6)$-dimensional family of {\it complete} minimal submanifolds of codimension two. 

\begin{theorem}\label{theorem-SUsSp}
Let $n\geq 2$ and $a,b\in\C^{2n}$ be linearly independent vectors such that $J_n \bar a$ and $b$ are orthogonal. Let $\phi:\SUs{2n}/\Sp n \to \C$ be the eigenfunction given by
$$
\phi(z\cdot\Sp n) = \trace (ab^tzJ_nz^t).
$$
Then the inverse image $\phi\inv(\{0\})$ is a complete minimal submanifold of $\SUs{2n}/\Sp n$ of codimension two.
\end{theorem}

The proofs of these four results are provided below. 
The readers interested in further details are referred to \cite{Lar-MSc}.

%%%%%%%%%%%%%%%%%%%%%%%%%%%%%%%%%%%%%%%%%%%
\section{Eigenfunctions and Eigenfamilies}
\label{section-eigenfunctions}
%%%%%%%%%%%%%%%%%%%%%%%%%%%%%%%%%%%%%%%%%%%

Let $(M,g)$ be an $m$-dimensional Riemannian manifold and $T^{\cn}M$ be the complexification of the tangent bundle $TM$ of $M$. We extend the metric $g$ to a complex bilinear form on $T^{\cn}M$.  Then the gradient $\nabla\phi$ of a complex-valued function $\phi:(M,g)\to\cn$ is a section of $T^{\cn}M$.  In this situation, we have the well-known complex linear {\it Laplace-Beltrami operator} (alt. {\it tension field}) $\tau$ on $(M,g)$.  In local coordinates this satisfies 
$$
\tau(\phi)=\Div (\nabla \phi)=\sum_{i,j=1}^m\frac{1}{\sqrt{|g|}} \frac{\partial}{\partial x_j}
\left(g^{ij}\, \sqrt{|g|}\, \frac{\partial \phi}{\partial x_i}\right).
$$
For two complex-valued functions $\phi,\psi:(M,g)\to\cn$ we have the following well-known fundamental relation
\begin{equation*}\label{equation-basic}
\tau(\phi\cdot \psi)=\tau(\phi)\cdot\psi +2\,\kappa(\phi,\psi)+\phi\cdot\tau(\psi),
\end{equation*}
where the symmetric complex bilinear {\it conformality operator} $\kappa$ is given by $$\kappa(\phi,\psi)=g(\nabla \phi,\nabla \psi).$$  Locally this satisfies 
$$\kappa(\phi,\psi)=\sum_{i,j=1}^mg^{ij}\cdot\frac{\partial\phi}{\partial x_i}\frac{\partial \psi}{\partial x_j}.$$

The naming of the operator $\kappa$ comes from the fact that $\kappa (\phi,\phi)=0$ if and only if 
$$\kappa (\phi,\phi)=|\nabla u|^2-|\nabla v|^2+2i\cdot  g(\nabla u,\nabla v)=0.$$

\begin{definition}\cite{Gud-Sak-1}\label{definition-eigenfamily}
Let $(M,g)$ be a Riemannian manifold. Then a complex-valued function $\phi:M\to\cn$ is said to be a {\it $(\lambda,\mu)$-eigenfunction} if it is eigen both with respect to the Laplace-Beltrami operator $\tau$ and the conformality operator $\kappa$ i.e. there exist complex numbers $\lambda,\mu\in\cn$ such that $$\tau(\phi)=\lambda\cdot\phi\ \ \text{and}\ \ \kappa(\phi,\phi)=\mu\cdot \phi^2.$$	
A set $\E =\{\phi_i:M\to\cn\ |\ i\in I\}$ of complex-valued functions is said to be a {\it $(\lambda,\mu)$-eigenfamily} on $M$ if there exist complex numbers $\lambda,\mu\in\cn$ such that for all $\phi,\psi\in\E$ we have 
$$\tau(\phi)=\lambda\cdot\phi\ \ \text{and}\ \ \kappa(\phi,\psi)=\mu\cdot \phi\,\psi.$$ 
\end{definition}
\medskip

For the standard odd-dimensional round spheres we have the following eigenfamilies based on the classical real-valued spherical harmonics.

\begin{example}\cite{Gud-Mun-1}\label{example-basic-sphere} 
Let $S^{2n-1}$ be the odd-dimensional unit sphere in the standard Euclidean space $\cn^{n}\cong\rn^{2n}$ and define $\phi_1,\dots,\phi_n:S^{2n-1}\to\cn$ by
$$\phi_j:(z_1,\dots,z_{n})\mapsto \frac{z_j}{\sqrt{|z_1|^2+\cdots +|z_n|^2}}.$$  Then the tension field $\tau$ and the conformality operator $\kappa$ on $S^{2n-1}$ satisfy
$$\tau(\phi_j)=-\,(2n-1)\cdot\phi_j\ \ \text{and}\ \ \kappa(\phi_j,\phi_k)=-\,1\cdot \phi_j\cdot\phi_k.$$
\end{example}
\medskip

For the standard complex projective space $\cn P^n$ we similarly have a complex multidimensional eigenfamily.

\begin{example}\cite{Gud-Mun-1}\label{example-basic-projective-space}
Let $\cn P^n$ be the standard $n$-dimensional complex projective space. For a fixed integer $1\le\alpha < n+1$ and some $1\le j\le\alpha < k\le n+1$  define the function $\phi_{jk}:\cn P^n\to\cn$ by
$$\phi_{jk}:[z_1,\dots,z_{n+1}]\mapsto \frac
{z_j\cdot\bar z_k}{z_1\cdot \bar z_1+\cdots + z_{n+1}\cdot \bar z_{n+1}}.$$  
Then the tension field $\tau$ and the conformality operator $\kappa$ on $\cn P^n$ satisfy
$$\tau(\phi_{jk})=-\,4(n+1)\cdot\phi_{jk}\ \ \text{and}\ \ \kappa(\phi_{jk},\phi_{lm})=-\,4\cdot \phi_{jk}\cdot\phi_{lm}.$$
\end{example}

In recent years, explicit eigenfamilies of complex-valued functions have been found on all the classical compact Riemannian symmetric spaces. For those relevant for this work see Table \ref{table-eigenfamilies}.

\renewcommand{\arraystretch}{2}
\begin{table}[h]\label{table-eigenfamilies}
	\makebox[\textwidth][c]{
		\begin{tabular}{cccc}
			\midrule
			\midrule
$G/K$	& $\lambda$ & $\mu$ & Eigenfunctions \\
\midrule
\midrule
$\SU n/\SO n$ & $-\,\frac{2(n^2+n-2)}{n}$& $-\,\frac{4(n-1)}{n}$ & 
see \cite{Gud-Sif-Sob-2} \\
\midrule
$\Sp n/\U n$ & $-\,2(n+1)$ & $-\,2$ & 
see \cite{Gud-Sif-Sob-2} \\
\midrule
$\SO{2n}/\U n$ & $-\,2(n-1)$ & $-1$ & 
see \cite{Gud-Sif-Sob-2} \\
\midrule
$\SU{2n}/\Sp n$ & $-\,\frac{2(2n^2-n-1)}{n}$ & $-\,\frac{2(n-1)}{n}$ & 
see \cite{Gud-Sif-Sob-2} \\
\midrule
\midrule
\end{tabular}	
}
\bigskip
\caption{Eigenfamilies on the relevant classical compact irreducible Riemannian symmetric spaces.}
\label{table-eigenfamilies}	
\end{table}
\renewcommand{\arraystretch}{1}

We conclude this section with the following two results, particularly useful in the above-mentioned situations of compact Riemannian symmetric spaces.

\begin{proposition}\cite{Gud-Sif-Sob-2}\label{proposition-lift}
Let $\pi:(\hat M,\hat g)\to (M,g)$ be a harmonic Riemannian submersion between Riemannian manifolds. Further let $\phi:(M,g)\to\C$ be a smooth function and $\hat\phi:(\hat M,\hat g)\to\C$ be the composition $\hat\phi=\phi\circ\pi$. Then the corresponding tension fields $\tau$ and conformality operators $\kappa$  satisfy
$$\tau(\hat\phi)=\tau(\phi)\circ\pi\ \ \text{and}\ \
\kappa(\hat\phi,\hat\psi) = \kappa(\phi,\psi)\circ\pi.$$
\end{proposition}

\begin{proof} The arguments needed here can be found in \cite{Gud-Sif-Sob-2}.
\end{proof}

In the sequel, we shall apply the following immediate consequence of Proposition \ref{proposition-lift}.

\begin{corollary}{\rm \cite{Gud-Sif-Sob-2}}\label{corollary-lift-eigen}
Let $\pi:(\hat M,\hat g)\to (M,g)$ be a harmonic Riemannian submersion.  For a complex-valued smooth function $\phi:(M,g)\to\C$ let $\hat \phi:(\hat M,\hat g)\to\C$ be the composition $\hat \phi=\phi\circ\pi$. Then the following statements are equivalent
\begin{enumerate}
\item[(i)] 
$\phi:M\to\C$ is a $(\lambda,\mu)$-eigenfunction on $M$,
\item[(ii)] 
$\hat\phi:\hat M\to\C$ is a $(\lambda,\mu)$-eigenfunction on $\hat M$.
\end{enumerate}
\end{corollary}

%%%%%%%%%%%%%%%%%%%%%%%%%%%%%%%%%%%%%%%%%%%
\section{Minimal Submanifolds via Eigenfunctions}
\label{section-minimal-submanifolds}
%%%%%%%%%%%%%%%%%%%%%%%%%%%%%%%%%%%%%%%%%%%

The recent paper \cite{Gud-Mun-1} provides an application of complex-valued eigenfunctions.  This is a method for constructing minimal submanifolds of codimension two.

\begin{theorem}
{\cite{Gud-Mun-1}}\label{theorem-Gud-Mun-1}
Let $\phi:(M,g)\to\cn$ be a complex-valued eigenfunction on a Riemannian manifold, such that $0\in\phi(M)$ is a regular value for $\phi$.  Then the fibre $\phi^{-1}(\{0\})$ is a minimal submanifold of $M$ of codimension two.
\end{theorem}

The main aim of our work is to apply Theorem \ref{theorem-Gud-Mun-1} in several of the interesting cases when the manifold $(M,g)$ is one of the classical non-compact Riemannian symmetric spaces.

The next result, from Riedler and Siffert's paper \cite{Rie-Sif} supplies us with a straightforward way of checking whether an eigenfunction, on a {\it compact} and connected Riemannian manifold, attains the required value $0\in\cn$.

\begin{theorem}{\rm \cite{Rie-Sif}}\label{phi0}
Let $(M,g)$ be a compact and connected Riemannian manifold and let $\phi:M\rightarrow\cn$ be a $(\lambda,\mu)$-eigenfunction not identically zero. Then the following are equivalent.
\begin{enumerate}
\item $\lambda=\mu$.
\item $\rvert\phi\rvert^2$ is constant.
\item $\phi(x)\neq0$ for all $x\in M.$
\end{enumerate}
\end{theorem}

As an obvious consequence we have the following.

\begin{corollary}\label{corollary-phi0}
If $\phi:M\rightarrow\cn$ is a complex-valued $(\lambda,\mu)$-eigenfunction on a compact and connected Riemannian manifold $(M,g)$ such that $\lambda\neq\mu,$ then there exists $x\in M$ such that $\phi(x)=0$.
\end{corollary}

In the cases of {\it non-compact} Riemannian symmetric space, here under investigation, this is clearly not available, and we therefore need to approach this is a different manner.  For this see Lemma \ref{lemma-regular-orthogonal}.
 
%%%%%%%%%%%%%%%%%%%%%%%%%%%%%%%%%%%%%%%%%%%%%%%%%%%%%%%
\section{Riemannian Symmetric Spaces and Their Duality}
\label{section-duality}
%%%%%%%%%%%%%%%%%%%%%%%%%%%%%%%%%%%%%%%%%%%%%%%%%%%%%%%

Let $(M,g)$ be a connected {\it non-compact} Riemannian symmetric space.  Then $M$ is isometric to the quotient $G/K$ under a suitable left-invariant metric.  Here $G$ is the connected component of the isometry group of $(M,g)$ containing the neutral element and $K$ is a maximal compact subgroup of $G$.  For this we have the orthogonal Cartan decomposition
$$
\g=\k\oplus \p
$$ 
of the Lie algebra $\g$ of $G$, where $\k$ is the Lie algebra of the compact subgroup $K$ of $G$ and $\p$ its orthogonal complement in $\g$. Let $G^\cn$ be the complexification of $G$.  Then $G^\cn$ is a Lie group with Lie algebra 
$$
\g^\cn=\g\oplus i\, \g = \k\oplus \p\oplus i\,\k\oplus i\,\p
$$ 
Let $U$ be the subgroup of $G^\cn$ with the Lie subalgebra $\mathfrak{u} =\k\oplus i\,\p$ of $\g^\cn$.  Then $U$ is compact and the quotient $U/K$ is a Riemannian symmetric space called the {\it compact dual} of the non-compact $M=G/K$. The corresponding natural projections $\pi_G:G \to G/K$ and $\pi_U: U \to U/K$ are Riemannian submersions.  For the general theory of symmetric spaces we refer to the standard work \cite{Hel} of Helgason.
\smallskip 

Let $W$ and $W^*$ be open subsets of $G/K$ and $U/K$, respectively. Two real-analytic functions 
$$
\phi:W \to \C,\quad\phi^*:W^* \to \C
$$ 
are said to be {\it dual} if there is an open subset $W^\C$ of the shared complexified Lie group $G^\C = U^\C$ and an analytic function $\phi^\C:W^\C \to \C$ such that 
$$
W = \pi_G(W^\C\cap G),\quad W^* = \pi_U(W^\C\cap U)
$$ 
and 
$$
\phi = \pi_G\circ\phi^\C|_G,\quad \phi^* = \pi_U\circ\phi^\C|_U.
$$

For the above situation we have the following useful result.

\begin{theorem}\label{theorem-duality}\cite{Gud-Sve-2}
Let $G/K$ and $U/K$ be a dual pair of Riemannian symmetric spaces and $W$ an open subset of $G$. If $\phi,\psi:W\to\C$ are real analytic functions and $\phi^*,\psi^*:W^* \to \C$ are their duals, then we have
$$\tau(\phi^*) = -\tau(\phi)^*\ \ \text{and}\ \ 
\kappa(\phi^*,\psi^*) = -\kappa(\phi,\psi)^*.$$
\end{theorem}
\begin{proof}
The details can be found in the proof of Theorem \ref{theorem-duality} in \cite{Gud-Sve-2}, see also \cite{Gud-Mon-Rat-1}.
\end{proof}

\begin{corollary}\label{corollary-dual-eigenfamilies}\cite{Gud-Sve-2}
Let $G/K$ and $U/K$ be a dual pair of Riemannian symmetric spaces. Then a collection $\E$ of complex-valued analytic functions on $G/K$ is a $(\lambda,\mu)$-eigenfamily if and only if the collection $\E^*$ of dual functions is a $(-\lambda,-\mu)$-eigenfamily on $U/K$.
\end{corollary}
\begin{proof}
This is a direct consequence of Theorem \ref{theorem-duality}.
\end{proof}

%%%%%%%%%%%%%%%%%%%%%%%%%%%%%%%%%%%%%%%%%%%%%%%%%%%%%%%

%%%%%%%%%%%%%%%%%%%%%%%%%%%%%%%%%%%%%%%%%%%%%%%%%%%%%%%
\section{The General Linear Group $\GLC n$}
\label{section-GLC}
%%%%%%%%%%%%%%%%%%%%%%%%%%%%%%%%%%%%%%%%%%%%%%%%%%%%%%%

In this section we now turn our attention to the concrete Riemannian matrix Lie groups embedded as subgroups of the complex general linear group.
\medskip

The group of linear automorphisms of $\cn^n$ is the complex general linear group $\GLC n=\{ z\in\cn^{n\times n}\,|\, \det z\neq 0\}$ of invertible $n\times n$ matrices with its standard representation 
$$z\mapsto
\begin{bmatrix}
	z_{11} & \cdots & z_{1n} \\
	\vdots & \ddots & \vdots \\
	z_{n1} & \cdots & z_{nn}
\end{bmatrix}.
$$
Its Lie algebra $\glc n$ of left-invariant vector fields on $\GLC n$ can be identified with $\cn^{n\times n}$ i.e.
 the complex linear space of $n\times n$ matrices.  We equip $\GLC n$ with its natural left-invariant Riemannian metric $g$ induced by the standard Euclidean inner product $g:\glc n\times\glc n\to\rn$ on its Lie algebra $\glc n$ satisfying
$$g(Z,W)\mapsto \Re\,\trace\, (Z\cdot W^*).$$ 
For $1\le i,j\le n$, we shall by $E_{ij}$ denote the element of $\rn^{n\times n}$ satisfying
$$(E_{ij})_{kl}=\delta_{ik}\delta_{jl}$$ and by $D_t$ the diagonal
matrices $D_t=E_{tt}.$ For $1\le r<s\le n$, let $X_{rs}$ and
$Y_{rs}$ be the matrices satisfying
$$X_{rs}=\frac 1{\sqrt 2}(E_{rs}+E_{sr}),\ \ Y_{rs}=\frac
1{\sqrt 2}(E_{rs}-E_{sr}).$$
The real vector space $\glc n$ then has the canonical orthonormal basis $\B^\cn=\B\cup i\B$, where 
$$\B=\{Y_{rs}, X_{rs}\,|\, 1\le r<s\le n\}\cup\{D_{t}\,|\, t=1,2,\dots,n\}.$$
\vskip .1cm

Let $G$ be a classical Lie subgroup of $\GLC n$ with Lie algebra $\g$ inheriting the induced left-invariant Riemannian metric, which we shall also denote by $g$.  In the cases considered in this paper, $\B_{\g}=\B^\cn\cap\g$ will be an orthonormal basis for the subalgebra $\g$ of $\glc n$.  By employing the Koszul formula for the Levi-Civita connection $\nabla$ on $(G,g)$, we see that for all $Z,W\in\B_{\g}$ we have
\begin{eqnarray*}
	g(\nab ZZ,W)&=&g([W,Z],Z)\\
	&=&\Re\,\trace\, ((WZ-ZW)Z^*)\\
	&=&\Re\,\trace\, (W(ZZ^*-Z^*Z))\\
	&=&0.
\end{eqnarray*}

If $Z\in\g$ is a left-invariant vector field on $G$ and $\phi:U\to\cn$ is a local complex-valued function on $G$ then the $k$-th order derivatives $Z^k(\phi)$ satisfy
\begin{equation*}\label{equation-diff-Z}
	Z^k(\phi)(p)=\frac {d^k}{ds^k}\bigl(\phi(p\cdot\exp(sZ))\bigr)\Big|_{s=0}.
\end{equation*}
\smallskip 

\noindent
This implies that the tension field $\tau$ and the conformality operator $\kappa$ on $G$ fulfill 
\begin{equation*}\label{equation-tau}
	\tau(\phi)
	=\sum_{Z\in\B_\g}\bigl(Z^2(\phi)-\nab ZZ(\phi)\bigr)
	=\sum_{Z\in\B_\g}Z^2(\phi),
\end{equation*}	
\begin{equation*}\label{equation-kappa}
	\kappa(\phi,\psi)=\sum_{Z\in\B_\g}Z(\phi)\cdot Z(\psi),
\end{equation*}
where $\B_\g$ is the orthonormal basis $\B^\cn\cap\g$ for the Lie algebra $\g$.

\begin{lemma}\label{lemma-regular-orthogonal}
Let $G$ be a classical Lie subgroup of $\GLC n$ with Lie algebra $\g$ generated by a subset $\B_{\g}$ of the orthonormal basis $\B^\cn$ for $\glc n$. Let $\psi:\C^{n\times n} \to \C$ be given by
$$
\psi(x) = \trace (ab^txBx^t),$$
where $a,b\in\C^n$ and $B\in\GLC n$ is either symmetric or skew-symmetric. Finally, let $\phi = \psi|_G$ be the restriction of $\psi$ to $G$. Then a point $x\in G$ satisfies
\begin{enumerate}
\item  $\phi(x) = 0$ if and only if $(Bx^tb, x^ta) = 0$ where $(\cdot, \cdot)$ is the standard bilinear form on $\C^n$, and 
\item $d\phi(x) = 0$ if and only if the matrix
$$
Bx^tab^tx\in (\g^\C)^\perp\subseteq \C^{n\times n}
$$
i.e. it lies in the orthogonal complement of the complexification of $\g$ with respect to the standard Euclidean inner product $\langle\cdot,\cdot\rangle$.
\end{enumerate}
\end{lemma}

\begin{proof}
The first statement follows directly from the fact that $B$ is either symmetric or skew-symmetric and thus
$$
\phi(x) = \trace (ab^txBx^t) = \pm\,\trace (x^ta(Bx^tb)^t) = \pm\,(x^ta, Bx^tb).
$$

For the second statement, the differential $d\phi(x)$ vanishes if and only if we have $X(\phi)(x) = 0$ for all $X\in\g$. Let $X\in \g$, then
\begin{eqnarray*}
X(\phi)(x) 
&=& \frac{d}{dt}\trace (A(x\cdot\exp(tX)) B (x\cdot\exp(tX))^t)|_{t=0} \\
&=& \trace (AxXBx^t) + \trace (AxB(xX)^t)\\
&=& \trace (AxXBx^t) + \trace (xXB^tx^tA^t) \\
&=& \trace (Bx^tAx\cdot X) \pm \trace (Bx^tA^tx\cdot X) \\
&=& \trace (Bx^t(A\pm A^t)x\cdot X)\\
&=& 2\,\trace (Bx^tAx\cdot X).
\end{eqnarray*}
In the last line we use the fact that the expression $\trace(AxBx^t)$ depends only on the (skew) symmetric part of $A$ owing to the (skew) symmetry of $B$.

Notice that $\B_{\g}$ consists of matrices which are either completely real or completely imaginary, as well as being symmetric or skew-symmetric. Suppose first that $X\in\B_{\g}$ is real, then
\begin{eqnarray*}
\trace (Bx^tAx\cdot X) 
&=& \Re\,\trace (Bx^tAx\cdot X) + i\,\Im\,\trace (Bx^tAx\cdot X)\\
&=& \pm\,\Re\,\trace (Bx^tAx\cdot X^t) \pm\, i\,\Re\,\trace (Bx^tAx\cdot iX^t)\\
&=& \pm\,\Re\,\trace(Bx^tAx\cdot X^*) \pm\, i\,\Re\,\trace (Bx^tAx\cdot (iX)^*)\\
&=& \pm\,\langle Bx^tAx, X\rangle\pm\, i\langle Bx^tAx, iX\rangle,
\end{eqnarray*}
and this vanishes if and only if 
$$
\langle Bx^tAx, X\rangle = \langle Bx^tAx, iX\rangle = 0.
$$
Similarly, if $X\in\B_{\g}$ is purely imaginary then $\trace (Bx^tAx\cdot X) = 0$ if and only if
$$
\langle Bx^tAx, X\rangle = \langle Bx^tAx, iX\rangle = 0.
$$
As a consequence, $d\phi(x) = 0$ implies that
$$
\langle Bx^tAx, Z\rangle = 0
$$
for all $Z\in\g^\cn$.
\end{proof}

%%%%%%%%%%%%%%%%%%%%%%%%%%%%%%%%%%%%%%%%%%%
\section{The Symmetric Space $\SLR n/\SO n$}
\label{section-SUn-SOn}
%%%%%%%%%%%%%%%%%%%%%%%%%%%%%%%%%%%%%%%%%%%

The purpose of this section is to prove Theorem \ref{theorem-SLRSO} and thereby construct a new multidimensional family of complete minimal submanifolds of the homogeneous quotient manifold $\SLR n/\SO n$.  This  carries the structure of a non-compact Riemannian symmetric space. It is well-known that the natural projection $\pi:\SLR n\to\SLR n/\SO n$ is a Riemannian submersion.  This means that we can apply Corollary \ref{corollary-lift-eigen} in this situation.
\smallskip

The non-compact special linear group $\SLR n$ is given by
$$\SLR{n}=\{x\in\GLR{n}\,|\,\det x =1\}.$$
The Lie algebra $\slr n$ of $\SLR n$ is the set of real  traceless matrices
$$
\slr n=\{X\in\glr n\, |\, \trace X=0\}.
$$
The special orthogonal group $\SO n$ is the maximal compact subgroup of $\SLR n$ given by
$$
\SO{n}=\{x\in\SLR{n}\, |\, x\cdot x^t=I_n\}.
$$ The Lie algebra $\so n$ of $\SO n$ is the set of real  skew-symmetric matrices
$$
\so n=\{X\in\glr n\, |\, X+X^t=0\}.
$$ 
For the Lie algebra $\slr n$ we have the orthogonal decomposition 
$$
\slr n=\so n\oplus \p,
$$
where the subspace $\p$ satisfies 
$\p=\{ X\in\slr n\,|\, X^t=X\}$.

\begin{proposition}\label{proposition-eigenfamily_slr}
For a non-zero element $a\in\C^n$ let the complex-valued function $\phi_a:\SLR n/\SO n\to\C$ be defined by
$$
\phi_a(x\cdot\SO n) = \trace(aa^txx^t).
$$ 
Then $\phi_a$ is a well-defined eigenfunction with
$$
\lambda = 2\cdot\frac{n^2+n-2}{n}\ \ \text{and}\ \ \mu = 4\cdot\frac{n-1}{n}.
$$
\end{proposition}

\begin{proof}
This is a direct consequence of Proposition 4.1 of \cite{Gud-Sif-Sob-2} and the duality presented in Theorem \ref{theorem-duality}.
\end{proof}

We will now prove our first main result formulated in Theorem \ref{theorem-SLRSO}.

\begin{proof}{(Theorem \ref{theorem-SLRSO})}
We begin by lifting $\phi$ to the $\SO n$-invariant map $\hat\phi:\SLR n \to \cn$ in the obvious way, namely
$$\hat\phi(x) = \trace (aa^txx^t).$$
Here we see that $G = \SLR n$ and $\hat\phi$ satisfy the conditions of Lemma \ref{lemma-regular-orthogonal} with $a = b$ and $B = I_n$. By assumption, we have
$$a = a_1 + ia_2,$$
where $a_1,a_2\in\R^n$ are linearly independent, so we can form a basis 
$$
\{a_1,a_2,\dots, a_n\}
$$ for $\R^n$. Since $n > 2$ we have at least one degree of freedom with which to ensure that the matrix
$$
y = (a_1,a_2,\dots,a_n) \in \SLR n
$$
has determinant $1$. Then letting
$$
x = (y\inv)^t \in \SLR n
$$
we have
$$
x^ta = y\inv a_1 + iy\inv a_2 = e_1 + ie_2,
$$
which satisfies
$$
(x^ta, x^ta) = (e_1,e_1) - (e_2, e_2) = 0.
$$
Hence, the fiber $\phi^{-1}(\{0\})$ is non-empty by Lemma \ref{lemma-regular-orthogonal}.
\smallskip

Let us now assume that $x\in\phi^{-1}(\{0\})$ is a critical point. The complexification of the real Lie algebra $\slr n$ is $\slc n$ whose orthogonal complement in $\glc n$ is the complex one dimensional subspace spanned by the identity matrix $I_n$. Again using Lemma \ref{lemma-regular-orthogonal}, we see that $x$ is singular if and only if
$$x^taa^tx = \alpha\cdot I_n$$
for some non-zero $\alpha\in \C$. However, the matrix $x^taa^tx$ always has rank $1$, whereas $\alpha\cdot I_n$ has rank $n > 2$. This gives us a contradiction. 
The result now follows from Theorem \ref{theorem-Gud-Mun-1}.

We also remark that for $n\geq3$, the linear independence condition on $a_1,a_2$ is necessary and sufficient. This is because if $x^ta$ is isotropic then in particular $x^ta_1$ and $x^ta_2$ are linearly independent, which implies that $a_1,a_2$ must have been so. 
\end{proof}

\begin{example}
 Let us now consider a particular choice of $\phi$ in the case when $n=3$. The Riemannian symmetric space $\SLR 3/\SO 3$ is a 5-dimensional manifold, so we are expecting Theorem \ref{theorem-SLRSO} to result in 3-dimensional complete minimal submanifolds.
   
Let $a^t = (1,i,0)$ and define the map $\phi:\SLR 3/\SO 3\to\C$ by
\begin{eqnarray*}
& &\phi(x\cdot\SO 3) \\
&=& \trace(aa^txx^t)\\
&=& x_{11}^2 + x_{12}^2 + x_{13}^2 + 2i(x_{11}x_{21} + x_{12}x_{22} + x_{13}x_{33}) - x_{21}^2 - x_{22}^2 - x_{23}^2\\
&=& \|x_1\|^2-\|x_2\|^2 + 2i\,\langle x_1, x_2\rangle,
\end{eqnarray*} 
where $x_1,x_2,x_3$ are the rows of $x$. Then $\phi(x\cdot\SO 3)$ is zero if and only if the first two rows of $x$ are orthogonal and of equal length, or equivalently $x_1+ix_2\in\C^3$ is isotropic. By Theorem \ref{theorem-SLRSO}, the preimage  
$$\phi^{-1}(\{0\}) = \{x\cdot \SO 3 \ |\ (x_1 + ix_2)\in\cn^3 \ \text{is isotropic}\}$$  
is a complete minimal submanifold of $\SLR 3/\SO 3$. The following provides a better picture of its geometry.

    Given a coset $x\cdot\SO 3\in\phi^{-1}(\{0\})$, we can define a canonical representative $\tilde x$ as follows. Let $u = \|x_1\| = \|x_2\| \in \R^+$ be the shared length of the first two rows of $x$, and put $y_1 = x_1/u$ and $y_2 = x_2/u$. Then $\langle y_1, y_2\rangle = u^{-2}\langle x_1, x_2\rangle = 0$. There is a unique $y_3\in\R^3$ such that $\{y_1,y_2,y_3\}$ is an oriented orthonormal basis for $\R^3$. Put $y = (y_1^t,y_2^t,y_3^t)\in\SO 3$. Then, using the fact that $\det(\tilde x) = 1$, we obtain 
$$\tilde x = x\cdot y = \left(\begin{array}{ccc}
        x_1\cdot y_1^t & x_1\cdot y_2^t & x_1\cdot y_3^t \\
        x_2\cdot y_1^t & x_2\cdot y_2^t & x_2\cdot y_3^t \\
        * & * & *
\end{array}\right) =\left(\begin{array}{ccc}
        u & 0 & 0 \\
        0 & u & 0 \\
        v & w & u^{-2}
\end{array}\right)$$
for some $v,w\in \R$. The uniqueness of $y$ shows that each coset has a unique representative of this form. Hence, the parametrisation $\R^+\times\R^2 \to \phi^{-1}(\{0\})$ defined by
$$
(u,v,w) \mapsto \left(\begin{array}{ccc}
        u & 0 & 0 \\
        0 & u & 0 \\
        v & w & u^{-2}
\end{array}\right)\cdot\SO3
$$
is bijective. Given that the matrix $y$ depends smoothly on $x$ this is in fact a diffeomorphism. Thus, $\phi^{-1}(\{0\})$ is a minimal submanifold of $\SLR 3/\SO 3$ diffeomorphic to $\R^3$.
\end{example}

%%%%%%%%%%%%%%%%%%%%%%%%%%%%%%%%%%%%%%%%%%%
\section{The Symmetric Space $\SpR{n}/\U n$}
\label{section-Spn-Un}
%%%%%%%%%%%%%%%%%%%%%%%%%%%%%%%%%%%%%%%%%%%

In this section we prove Theorem \ref{theorem-SpRU}, which yields a new multidimensional family of complete minimal submanifolds of the non-compact Riemannian symmetric space $\SpR{n}/\U n$. 
\smallskip 

The real symplectic group $\SpR n$ is defined with
$$\SpR n = \{x\in\GLR {2n}\, |\, xJ_nx^t = J_n\},$$
where
$$
J_n = \left(\begin{array}{cc}
0 & I_n \\
-I_n & 0 
\end{array}\right)
$$
defines the standard skew-symmetric form on $\C^{2n}$. The Lie algebra $\spR n$ of $\SpR n$ is given by
$$
\spR n=\{X\in\rn^{2n\times 2n}\,|\,XJ_n+J_nX^t=0\}.
$$
The unitary group $\U n=\{z=x+iy\in\cn^{n\times n}\,|\,z\cdot z^*\}$ can be embedded into the real symplectic group $\SpR n$ as its maximal compact subgroup by 
$$
z=x+iy \mapsto \left(\begin{array}{cc}
x & y \\
-y & x
\end{array}\right).$$

For the Lie algebra $\spR n$ we have the orthogonal decomposition
$$\spR n = \u n\oplus\p,$$
where the orthonormal basis $\B_{\u n}$ for the subalgebra $\u n$ is given by
\begin{eqnarray*}
\tfrac{1}{\sqrt2}\left(\begin{array}{cc}
Y_{rs} & 0 \\
0 & Y_{rs}
\end{array}\right),\ \
\tfrac{1}{\sqrt2}\left(\begin{array}{cc}
0 & X_{rs} \\
-X_{rs} & 0
\end{array}\right),\ \ 
\tfrac{1}{\sqrt2}\left(\begin{array}{cc}
0 & D_t \\
-D_t & 0
\end{array}\right).
\end{eqnarray*}
The orthonormal basis $\B_{\p}$ for the orthogonal complement $\p$ is generated by
$$\tfrac{1}{\sqrt2}\left(\begin{array}{cc}
X_{rs} & 0 \\
0 & -X_{rs}
\end{array}\right),\ \ 
\tfrac{1}{\sqrt2}\left(\begin{array}{cc}
0 & X_{rs} \\
X_{rs} & 0
\end{array}\right),
$$
$$
\tfrac{1}{\sqrt2}\left(\begin{array}{cc}
D_t & 0 \\
0 & -D_t
\end{array}\right),\ \
\tfrac{1}{\sqrt2}\left(\begin{array}{cc}
0 & D_t \\
D_t & 0
\end{array}\right),$$
where $1\leq r < s \leq n$ and $1\leq t\leq n$.

\begin{proposition}\label{proposition-eigenfunctions-spr}
Let $a\in\C^{2n}$ be a non-zero vector. Then the complex-valued map $\phi_a:\SpR n/\U n\to\C$ given by
$$
\phi_a(x\cdot\U n) = \trace(aa^t xx^t)
$$
is an eigenfunction with
$$
\lambda = 2\cdot(n+1)\ \ \text{and}\ \ \mu = 2.
$$
\end{proposition}
\begin{proof}
This is a direct consequence of Proposition 4.2 of \cite{Gud-Sif-Sob-2} and the duality presented in Theorem \ref{theorem-duality}.
\end{proof}

\begin{lemma}\label{lemma-symmetric-transitive}
For any pair of vectors $u,v\in\R^n$ with $u\neq 0$ there exists a symmetric matrix $s\in\Sym(\R^{n\times n})$ such that $su = v$.
\end{lemma}
\begin{proof}
If $v = 0$ then simply take $s = 0$. Otherwise, suppose that $|u| = |v| = 1$. If $u = -v$ then we can simply let $s=-I_n$, otherwise we have $v^tu > -1$ so we can define the unit vector 
$$
w = \frac{u + v}{\sqrt{2+2v^tu}}.
$$
Then the matrix
$$
s = -I_n + 2ww^t
$$
is symmetric and satisfies
$$
su = -u + \frac{u^tu+v^tu}{1+v^tu}(v+u) = v.
$$
For the general case, take $|u|\inv|v| s$ where $s$ is the matrix constructed as above from the normalised vectors $|u|\inv u$ and $|v|\inv v$.
\end{proof}

\begin{lemma}\label{lemma-spr-zero}
Let $n\geq 2$ and $a=a_1+ia_2\in\C^{2n}$ be such that $a_1,a_2\in\R^{2n}$ are linearly independent. Then there exists an element $x\in\SpR n$ such that the image $x^ta$ is isotropic.
\end{lemma}
\begin{proof}
Note that the Lie group $\SpR n$ contains all matrices of the form
$$
D = \left(\begin{array}{cc}
d\inv & 0 \\
0 & d^t
\end{array}\right)
\quad\text{and}\quad 
S = \left(\begin{array}{cc}
I_n & s \\
0 & I_n
\end{array}\right),
$$
where $d\in\GLR n$ and $s\in\Sym(\R^{n\times n})$ are arbitrary. See Proposition 1.1.1 in \cite{Hab-Hab}.

    We write
    $$a = b + ic = \left(\begin{array}{c}
         b_1\\
         b_2
    \end{array}\right) + i\left(\begin{array}{c}
         c_1\\
         c_2
    \end{array}\right),
$$
where $b_1, b_2, c_1, c_2\in\R^n$. We may assume that $b_1\neq 0$ (otherwise $b_2\neq 0$ and a symmetric argument applies). Then by Lemma \ref{lemma-symmetric-transitive} there is some symmetric matrix $s_1\in\Sym_n(\R)$ satisfying $s_1b_1 = -b_2$. The block matrix
$$
S_1 = \left(\begin{array}{cc}
I_n & 0 \\
s_1 & I_n
\end{array}\right)\in\SpR n
$$
satisfies
$$
S_1b = \left(\begin{array}{c}
b_1\\
s_1b_1 + b_2
\end{array}\right)
 = \left(\begin{array}{c}
b_1\\
0
\end{array}\right),
\quad S_1c = \left(\begin{array}{c}
c_1\\
s_1c_1 + c_2
\end{array}\right) = 
\left(\begin{array}{c}
c_1\\
c_3
\end{array}\right) .
$$
If $c_3=0$ then $b_1$ and $c_1$ must be linearly independent, and thus they can be extended to a basis
$$
\{v_1 = b_1, v_2 = c_1, v_3, \dots, v_n\}
$$
for $\rn^n$, forming a matrix
$$
V = (v_1,\dots,v_n)\in\GLR n.
$$
Then the element
$$
D_1 = \left(\begin{array}{cc}
V\inv & 0 \\
0 & V^t\end{array}\right)\in\SpR n
$$
allows us to construct $x = (D_1S_1)^t\in\SpR n$ which has
$$
x^tb = \left(\begin{array}{c}
V\inv b_1\\
0
\end{array}\right) 
= \left(\begin{array}{c}
e_1\\
0
\end{array}\right),\quad 
x^tc = \left(\begin{array}{c}
V\inv c_1\\
0
\end{array}\right) = \left(\begin{array}{c}
e_2\\
0
\end{array}\right).$$
As a consequence $x^ta = x^tb+ix^tc = e_1 + ie_2$ is isotropic.
\smallskip
    
If instead $c_3\neq 0$, let $s_2\in\Sym(\R^{n\times n})$ be such that $s_2c_3 = -c_1$ and put 
$$
S_2 = \left(\begin{array}{cc}
        I_n & s_2 \\
        0 & I_n
    \end{array}\right)\in\SpR n,$$
    so that
    $$S_2S_1b = \left(\begin{array}{c}
         b_1\\
         0
    \end{array}\right),\ S_2S_1c = \left(\begin{array}{c}
         0\\
         c_3
    \end{array}\right).$$
    Now let $\lambda = \sqrt{|b_1|\inv|c_3|}$, put
    $$D_2 = \left(\begin{array}{cc}
        \lambda I_n & 0 \\
        0 & \lambda\inv I_n
    \end{array}\right)\in\SpR n$$
    and define $x = (D_2S_2S_1)^t\in\SpR n$. Then we have
    $$\langle x^tb, x^tc\rangle = \left\langle \left(\begin{array}{c}
         \lambda b_1\\
         0
    \end{array}\right), \left(\begin{array}{c}
         0\\
         \lambda\inv c_3
    \end{array}\right)\right\rangle = 0.$$
This implies that 
$$
|x^tb| = \lambda|b_1| = \sqrt{|b_1|\cdot|c_3|} = \lambda\inv |c_3| = |x^tc|
$$
so that
\begin{eqnarray*}
(x^ta, x^ta) 
&=& (x^tb + ix^tc, x^tb + ix^tc)\\
&=& |x^tb|^2 - |x^tc|^2 + 2i\langle x^tb, x^tc\rangle\\
&=& 0,
\end{eqnarray*}
in other words $x^ta$ is isotropic.
\end{proof}

After our preparations in Lemmas \ref{lemma-symmetric-transitive} and \ref{lemma-spr-zero} we are now ready to prove Theorem \ref{theorem-SpRU}.

\begin{proof}{(Theorem \ref{theorem-SpRU})}
Observe that the lift $\hat\phi:\SpR n \to \cn$ of $\phi$ satisfies the conditions of Lemma \ref{lemma-regular-orthogonal} with $a = b$ and $B = I_{2n}$. Thus, $x\in\SpR n$ is a zero of $\hat\phi$ if and only if $(x^ta, x^ta) = 0$ i.e. $x^ta\in\C^n$ is isotropic. The existence of zeros is then ensured by Lemma \ref{lemma-spr-zero}.

Let us now assume that $x\in\phi\inv(\{0\})$ satisfies $d\phi(x) = 0$. The complexification of $\spR n$ is simply $\spC n$, so by Lemma \ref{lemma-regular-orthogonal} we must have
$$
x^taa^tx\in\spC n^\perp.
$$
By examining the basis $\B_{\spC n}$ for $\spC n$ and using the fact that the matrix $x^taa^tx$ is symmetric, this forces it to be of the form
$$
x^taa^tx = 
\left(\begin{array}{cc}
U & V \\
-V & U 
\end{array}\right),
$$
where $U$ is symmetric and $V$ is skew-symmetric. Write
$x^ta = (b_1, b_2)^t$ with $b_1,b_2\in\C^n$ and observe that $U = b_1b_1^t = b_2b_2^t$ and $V = b_1b_2^t$. This means that $V$ is of rank at most 1. However, since every skew-symmetric matrix has even rank, this implies that $V = 0$ and thus $a = 0$, contradicting our assumptions. The result now follows from Theorem \ref{theorem-Gud-Mun-1}.
\end{proof}

\begin{example}
Let $n=2$ and consider the 6-dimensional Riemannian symmetric space $\SpR 2/\U 2$. We will now apply Theorem \ref{theorem-SpRU} in order to construct a complete four dimensional  minimal submanifold of $\SpR 2/\U 2$.

Take $a = e_1 + ie_2\in\C^{4}$. This satisfies the conditions of the theorem, and thus we have an eigenfunction $\hat\phi:\SpR 2 \to \C$ given by
$$
\hat\phi(x) = \trace(aa^txx^t).
$$
We now have $\hat\phi(x) = 0$ if and only if
$$
(x^ta, x^ta) = (x^te_1 + ix^te_2, x^te_1+ix^te_2) = (x_1+ix_2, x_1+ix_2) = 0
$$
i.e. $x_1+ix_2$ is an isotropic element, where $x_j$ denotes the $j$-th row of the matrix $x\in\SpR 2$. Thus, we have the complete four dimensional minimal submanifold
$$
\{x\cdot \U 2\, |\, x\in\SpR 2,\ |x_1|=|x_2|,\ \langle x_1,x_2\rangle = 0\}
$$
of $\SpR 2/\U 2$.
\end{example}

%%%%%%%%%%%%%%%%%%%%%%%%%%%%%%%%%%%%%%%%%%%
\section{The Symmetric Space $\SOs{2n}/\U n$}
\label{section-SO2n-Un}
%%%%%%%%%%%%%%%%%%%%%%%%%%%%%%%%%%%%%%%%%%%

The goal of this section is to prove Theorem \ref{theorem-SOsU}, which immediately provides us with a new multidimensional family of complete minimal submanifolds of the non-compact Riemannian symmetric space $\SOs{2n}/\U n$.
\smallskip 

The classical Lie group $\SOs{2n}$ is defined by
$$
\SOs{2n}\ = \{z\in\SOC{2n}\, |\, \bar zJ_nz^t = J_n\},
$$
where
$$
J_n = \left(\begin{array}{cc}
    0 & I_n \\
    -I_n & 0
\end{array}\right).
$$
The Lie algebra $\sos{2n}$ of the group $\SOs{2n}$ satisfies 
$$
\sos{2n} = \left\{\left(\begin{array}{cc}
    Z & W \\
    -\bar W & \bar Z
\end{array}\right)\in\C^{2n\times 2n}\, \Bigg|\, Z+Z^* = W-W^*=0\right\}.
$$ 
The unitary group $\U n$ appears as a maximal compact subgroup of $\SOs{2n}$ via the embedding
$$z=x+iy \mapsto \left(\begin{array}{cc}
    x & y \\
    -y & x
\end{array}\right)$$
and the quotient $\SOs{2n}/\U n$ is a classical non-compact Riemannian symmetric space.

The Lie algebra $\sos{2n}$ has the orthogonal Cartan decomposition
$$\sos{2n} = \u n\oplus \p,$$
and the orthonormal basis $\B_{\u n}$ for $\u n$ is given by
\begin{eqnarray*}
    \tfrac{1}{\sqrt2}\left(\begin{array}{cc}
        Y_{rs} & 0 \\
        0 & Y_{rs}
    \end{array}\right),\ \
    \tfrac{1}{\sqrt2}\left(\begin{array}{cc}
        0 & X_{rs} \\
        -X_{rs} & 0
    \end{array}\right),\ \
    \tfrac{1}{\sqrt2}\left(\begin{array}{cc}
        0 & D_t \\
        -D_t & 0
    \end{array}\right)
\end{eqnarray*}
and the orthonormal basis $\B_{\p}$ for the orthogonal subspace $\p$ by
\begin{eqnarray*}
    \tfrac{i}{\sqrt2}\left(\begin{array}{cc}
        Y_{rs} & 0 \\
        0 & -Y_{rs}
    \end{array}\right),\ \
    \tfrac{i}{\sqrt2}\left(\begin{array}{cc}
        0 & Y_{rs} \\
        Y_{rs} & 0
    \end{array}\right),
\end{eqnarray*}
where $1\leq r<s\leq n$ and $1\leq t \leq n$.

\begin{proposition}\label{proposition-eigenfunctions-SOs}
Let $a,b\in \C^{2n}$ be linearly independent vectors satisfying
$$
(a,a)\cdot(b,b) - (a,b) = 0,
$$
where $(\cdot,\cdot)$ is the standard complex bilinear form on $\C^{2n}$. Then the complex-valued map $\phi:\SOs{2n}/\U n\to\C$ given by
$$
\phi(z\cdot\U n) = \trace (ab^tzJ_nz^t)
$$
is an eigenfunction with
$$
\lambda = 2(n-1),\quad \mu=1.
$$
\end{proposition}

\begin{proof}
According to Proposition 1.1 in \cite{Geg-Gud-1}, if $a,b\in \C^{2n}$ are as in the statement, then the map $\phi^*:\SO{2n}/\U n \to \C$ given by
$$
\phi^*(x\cdot\U n) = \trace (\tfrac{1}{2}(ab^t-ba^t)xJ_nx^t)
$$
is an eigenfunction with
$$
\lambda = -2(n-1),\quad \mu = -1.
$$
There the authors are using the representation $\so{2n} = \u n\oplus i\,\p$ which is immediately dual to ours, and thus by Corollary \ref{corollary-dual-eigenfamilies} we see that the dual map
$$
\phi(z\cdot\U n) = \trace (\tfrac{1}{2}(ab^t-ba^t)zJ_nz^t)
$$
is an eigenfunction with 
$$
\lambda = 2(n-1),\quad \mu = 1.
$$
To finalise the argument, we show that this is the same as the map defined in the statement. Using the skew-symmetry of $J_n$ we have
\begin{eqnarray*}
\phi(z\cdot\U n) 
&=& \trace (\tfrac{1}{2}(ab^t-ba^t)zJ_n z^t)\\
&=& \tfrac{1}{2}(b^t(zJ_nz^t)a - a^t(zJ_nz^t)b) \\
&=& \tfrac{1}{2}b^t(zJ_nz^t - zJ_n^tz^t)a\\
&=& b^tzJ_nz^ta\\
&=& \trace (ab^tzJ_nz^t).
\end{eqnarray*}
This proves the statement.
\end{proof}

\begin{proof}{(Theorem \ref{theorem-SOsU})}
First, note that $\phi$ fulfills the conditions of Proposition \ref{proposition-eigenfunctions-SOs} and is thus an eigenfunction on $\SOs{2n}/\U n$. Let $\hat\phi:\SOs{2n}\to\C$ be the $\U n$-invariant lift of $\phi$ to the Lie group level. We will show that $\hat\phi$ has zero as a regular value, from which the same will follow for $\phi$ thanks to the $\U n$-invariance.

Observe also that $\hat\phi$ satisfies the conditions of Lemma \ref{lemma-regular-orthogonal} with $B = J_n$. Thus, we have $\hat\phi(z) = 0$ if and only if 
$$(J_nz^ta, z^tb) = 0.
$$
Let $z = I_{2n}$ be the identity matrix, then by assumption we have $(J_na, b) = 0$ and thus $\hat\phi(I_{2n}) = 0$. In particular the fiber $\phi^{-1}(\{0\})$ is non-empty.
\smallskip

Now suppose that $d\phi(z) = 0$ for some point $z\in\phi^{-1}(\{0\})$. The Lie algebra $\sos{2n}$ has the complexification $\soC{2n}$, which consists of the complex skew-symmetric matrices. Thus, by Lemma \ref{lemma-regular-orthogonal}, the matrix
$$
J_nz^tab^tz \in \soC n^\perp
$$
has to be symmetric. Being a symmetric matrix of rank $1$ means that it has the form $J_nz^tab^tz = ww^t$ for some $w\in\C^{2n}$. From this it follows that
\begin{eqnarray*}
\span_\C \{z^t b\} 
&=& \text{Im}(J_nz^tab^tz)^t\\
&=& \text{Im}(ww^t)^t \\
&=& \text{Im}(ww^t)\\
&=& \text{Im}(J_nz^tab^tz)\\
&=& \span_\C \{J_nz^ta\}
\end{eqnarray*}
and thus $J_nz^ta$ and $z^tb$ are linearly dependent. Given that $z^tb \neq 0$ there is a non-zero scalar $\lambda \in\C$ such that $J_nz^ta = \lambda z^tb$. Then, noting that the matrices $J_n, z^t \in\SOs{2n} \subset \SOC{2n}$ preserve the bilinear form $(\cdot, \cdot)$, we have
$$
\lambda^2(b,b) = (\lambda z^tb, \lambda z^tb) = (J_nz^ta, J_nz^ta) = (a,a) = 0
$$
and thus $(b,b) = 0$. This gives us a contradiction.
\end{proof}

\begin{example}
Let $n = 3$ and consider the six dimensional Riemannian symmetric space $\SOs 6/\U 3$. Using Theorem \ref{theorem-SOsU} we can construct a complete four dimensional minimal submanifold thereof.

We note that the vectors $a = e_1 + ie_2, b = e_6 \in \C^{6}$ fulfill the conditions of the theorem. The corresponding eigenfunction $\hat\phi:\SOs 6 \to \C$ defined by
$$
\hat\phi(z) = \trace (ab^tzJ_nz^t)
$$
satisfies $\hat\phi(z) = 0$ if and only if
$$
(J_3z^ta, z^tb) = -(\bar z^t(e_4+ie_5), z^te_6) = -\langle z_4, z_6\rangle + i\langle z_5, z_6\rangle = 0,
$$
which gives $\langle z_4, z_6\rangle = \langle z_5, z_6\rangle = 0$ where $\langle\cdot,\cdot\rangle$ is the standard Euclidean inner product on $\C^{6}$ and $z_j$ is the $j$-th row of the matrix $z\in\SOs 6$. Thus, we obtain the complete four dimensional minimal submanifold
$$
\{z\cdot\U 3 \, |\, z\in\SOs 6,\ \langle z_4, z_6\rangle = \langle z_5, z_6\rangle = 0\}
$$
of $\SOs 6/\U 3$.
\end{example}

%%%%%%%%%%%%%%%%%%%%%%%%%%%%%%%%%%%%%%%%%%%
\section{The Symmetric Space $\SUs{2n}/\Sp n$}
\label{section-SU2n-Spn}
%%%%%%%%%%%%%%%%%%%%%%%%%%%%%%%%%%%%%%%%%%%

In this last section we provide a proof of Theorem \ref{theorem-SUsSp}. Consequently, we obtain a new multidimensional family of complete minimal submanifolds of the non-compact Riemannian symmetric space $\SUs{2n}/\Sp n$.
\medskip 

The group $\SUs{2n}$ can be defined as $\SUs{2n} = \Us{2n}\cap\SLC {2n}$,
where
$$
\Us{2n} = \left\{\left(\begin{array}{cc}
    z & w \\
    -\bar w & \bar z
\end{array}\right)\, \Bigg| \, z,w\in\GLC n \right\}
$$
is the quaternionic general linear group. This is a generalisation of the standard two-dimensional complex  representation $\hn\to\C^{2\times 2}$
$$
q=z + jw \mapsto \left(\begin{array}{cc}
    z & w \\
    -\bar w & \bar z
\end{array}\right)
$$
of the quaternions $\hn$ on $\cn^2$.  Although there is no well-defined notion of a determinant for quaternionic matrices, this identification allows us to define $\SUs{2n}$ in terms of the usual complex determinant. Observe that we can also define $\SUs{2n}$ as
$$
\SUs{2n} = \{z\in \SLC {2n}\, |\, zJ_n = J_n\bar z\}.
$$

The quaternionic unitary group $\Sp n$ is the maximal compact subgroup of $\SUs{2n}$. It is the intersection of the unitary group $\U{2n}$ and the standard representation of the quaternionic general linear group $\Us{2n}$ in $\cn^{2n\times 2n}$ given by
\begin{equation*}\label{equation-Spn}
q=(z+jw)\mapsto 
\left(\begin{array}{cc}
    z & w \\
    -\bar w & \bar z
\end{array}\right)
\end{equation*}
The Lie algebra $\sp n$ of $\Sp n$ satisfies
$$
\sp{n}=\left\{\left(\begin{array}{cc}
    Z & W \\
    -\bar W & \bar Z
\end{array}\right)\in\cn^{2n\times 2n}
\, \Big|\, Z^*+Z=0,\ W^t-W=0\right\}.
$$ 

For the Lie algebra of $\Us{2n}$ we have the orthogonal Cartan decomposition 
$$\us{2n} = \sp n \oplus \p,$$
where an orthonormal basis $\B_{\sp n}$ for $\sp n$ is given by
\begin{eqnarray*}
    \tfrac{1}{\sqrt2}\left(\begin{array}{cc}
        X_{rs} & 0 \\
        0 & X_{rs}
    \end{array}\right),\ \ &\tfrac{1}{\sqrt{2}}\left(\begin{array}{cc}
        0 & X_{rs} \\
        -X_{rs} & 0
    \end{array}\right),\\
    \tfrac{1}{\sqrt2}\left(\begin{array}{cc}
        D_t & 0 \\
        0 & D_t
    \end{array}\right),\ \ &\tfrac{1}{\sqrt2}\left(\begin{array}{cc}
        0 & D_t \\
        -D_t & 0
    \end{array}\right),\\
    \tfrac{1}{\sqrt2}\left(\begin{array}{cc}
        Y_{rs} & 0 \\
        0 & Y_{rs}
    \end{array}\right),\quad &\tfrac{1}{\sqrt2}\left(\begin{array}{cc}
        0 & Y_{rs} \\
        -Y_{rs} & 0
    \end{array}\right),\\
    \tfrac{i}{\sqrt2}\left(\begin{array}{cc}
        Y_{rs} & 0 \\
        0 & -Y_{rs}
    \end{array}\right),\ \ &\frac{i}{\sqrt2}\left(\begin{array}{cc}
        0 & Y_{rs} \\
        Y_{rs} & 0
    \end{array}\right)
\end{eqnarray*}
and the orthonormal basis $\B_{\p}$ for the orthogonal subspace $\p$ by
\begin{eqnarray*}
\tfrac{i}{\sqrt2}\left(\begin{array}{cc}
        X_{rs} & 0 \\
        0 & -X_{rs}
    \end{array}\right),\ \ &\tfrac{i}{\sqrt2}\left(\begin{array}{cc}
        0 & X_{rs} \\
        X_{rs} & 0
    \end{array}\right),\\
    \tfrac{i}{\sqrt2}\left(\begin{array}{cc}
        D_t & 0 \\
        0 & -D_t
    \end{array}\right),\ \ &\tfrac{i}{\sqrt2}\left(\begin{array}{cc}
        0 & D_t \\
        D_t & 0
    \end{array}\right),
\end{eqnarray*}
where $1\leq r < s \leq n$ and $1\leq t \leq n$.
We then obtain the subalgebra $\sus{2n}$ of $\us{2n}$ by removing the effect of real scalar multiplication i.e. as the orthogonal complement of the real span of the identity matrix
$$\us{2n} = \span_\R\{I_{2n}\}\oplus \sus{2n}.$$

\begin{proposition}\label{proposition-eigenfunctions-SUs}
Let $a,b\in\C^{2n}$ be linearly independent. Then the\\ complex-valued map $\phi:\SUs{2n}/\Sp n\to\C$ given by
$$
\phi(z\cdot\Sp n) = \trace (ab^tzJ_nz^t)
$$
is an eigenfunction with
$$
\lambda = 2\cdot\tfrac{(n^2-n-1)}{n},\quad \mu=2\cdot\tfrac{(n-1)}{n}.
$$
\end{proposition}

\begin{proof}
According to Proposition 5.2 in \cite{Gud-Sif-Sob-2} the dual map\\ $\phi^*:\SU{2n}/\Sp n \to \C$ given by
$$
\phi^*(z\cdot\Sp n ) = \trace (\tfrac{1}{2}(ab^t-ba^t)zJ_nz^t)
$$
is an eigenfunction  with
$$
\lambda = -2\cdot\tfrac{(n^2-n-1)}{n}\quad\text{and}\quad \mu = -2\cdot\tfrac{(n-1)}{n}.
$$
The remaining argument is similar to that in the proof of Proposition \ref{proposition-eigenfunctions-SOs}.
\end{proof}

We can now provide a proof of Theorem \ref{theorem-SUsSp}.

\begin{proof}{(Theorem \ref{theorem-SUsSp})}
Let $\hat\phi:\SUs{2n} \to \C$ as usual be the $\Sp n$-invariant lift of $\phi$ to the group level. Observing that $\hat\phi$ satisfies the conditions for Lemma \ref{lemma-regular-orthogonal} with $B = J_n$, we see that $\hat\phi(z) = 0$ if and only if
$$
0 = (J_nz^t a, z^tb) = (\bar z^t J_na, z^tb) = (\overline{z^t (J_n \bar a)}, z^tb) = \langle z^t (J_n \bar a), z^tb\rangle,
$$
where $\langle\cdot,\cdot\rangle$ is the standard Euclidean inner product on $\C^{2n}$. By assumption, we thus have $I_{2n} \in \phi\inv(\{0\})$.
\smallskip

Let us now assume that $z\in\phi^{-1}(\{0\})$ is such that $d\phi(z) = 0$. The complexification of $\sus{2n}$ is $\slc {2n}$, whose orthogonal complement in $\glc {2n}$ is the complex 1-dimensional subspace spanned by $I_{2n}$. Then by Lemma \ref{lemma-regular-orthogonal} we must have
$$
J_nz^tab^tz = \alpha I_{2n}
$$
for some $\alpha \in \C\setminus\{0\}$. However, the former matrix is of rank $1$ whereas the latter is invertible, so we have a contradiction. The result now follows from Theorem \ref{theorem-Gud-Mun-1}.
\end{proof}

\begin{example}
Let $n=2$ and consider the 6-dimensional Riemannian symmetric space $\SUs 4/\Sp 2$. We will use Theorem \ref{theorem-SUsSp} in order to construct an explicit complete four dimensional minimal submanifold of this space.
    
Take the basis vectors $a = e_1$ and $b = e_2$ in $\C^4$. Note that $J_n\bar a = e_3$ is orthogonal to $b$ and thus the conditions of Theorem \ref{theorem-SUsSp} are satisfied. The eigenfunction $\hat\phi:\SUs 4 \to \C$ defined by
$$
\hat\phi(z) = \trace (ab^tzJ_nz^t)
$$
satisfies $\hat\phi(z) = 0$ if and only if
$$
(J_2z^te_1, z^te_2) = -(\bar z^te_3, z^te_1) = \langle z_3, z_1\rangle = 0,
$$
where $z_j$ denotes the $j$-th row of the matrix $z\in\SUs{2n}$. The induced eigenfunction $\phi:\SUs 4/\Sp 2 \to \C$ then gives rise to the complete four dimensional  minimal submanifold
$$
\{z\cdot \Sp 2 \, |\, z\in\SUs 4,\ \langle z_1, z_3\rangle = 0\}
$$
of $\SUs 4/\Sp 2$.
\end{example}

%%%%%%%%%%%%%%%%%%%%%%%%%%%%%%%%%%%%%%%%%%%%%%%%%%%%%%%%
\section{Acknowledgements}
%%%%%%%%%%%%%%%%%%%%%%%%%%%%%%%%%%%%%%%%%%%%%%%%%%%%%%%%

The authors are grateful to Thomas Jack Munn for useful discussions on this work.

%%%%%%%%%%%%%%%%%%%%%%%%%%%%%%%%%%%%%%%%%%%

\end{document}